\documentclass[11pt]{article}
\usepackage{amsmath}
\usepackage{times,theorem,latexsym,amssymb}
\usepackage[multiple]{footmisc}
\newcommand{\pagestart}{1}

\makeatletter

\def\@evenhead{\footnotesize\thepage\hfil\slshape\leftmark}%
\def\@oddhead{\footnotesize{\slshape\rightmark}\hfil\thepage}%

\numberwithin{equation}{section} \numberwithin{figure}{section}
\numberwithin{table}{section}

\renewcommand{\thefootnote}{\roman{footnote}}
\renewcommand{\thefootnote}{\Roman{footnote}}
\renewcommand{\thefootnote}{\arabic{footnote}}

\makeatother

\newenvironment{summary}{\vskip\baselineskip \noindent\small\bf Summary: \rm}%
{\vskip\baselineskip}

\newenvironment{proof}
{ \par\noindent{$ \mathsf{Proof\!:}$} }
{ \hspace{\stretch{1}}$\Box$}

\newtheorem{theorem}{Theorem }[section]
\newtheorem{corollary}[theorem]{Corollary}
\newtheorem{definition}[theorem]{Definition}
{\theorembodyfont{\rmfamily}\newtheorem{example}[theorem]{Example}}
\newtheorem{lemma}[theorem]{Lemma}
\newtheorem{proposition}[theorem]{Proposition}
{\theorembodyfont{\rmfamily}\newtheorem{remark}[theorem]{Remark}}

\title{Summability of Sequence of Random Variables}
\author{Jinlu Li, Robert Mendris}
\date{August 29, 2017}
\date{January 01, 2015}
\newcommand{\br}{\newline\noindent}

\newcommand{\R}{\mathbb{R}}
\newcommand{\Rii}{\mathbb{R}^{**}}
\newcommand{\ds}{\displaystyle} 

\newcommand{\ontop}[2]{\substack{{#1}\\{#2}}}
\newcommand{\suparrow}{\xrightarrow{\mathrm{a.s.}}}  
\newcommand{\abs}[1]{\left|{#1}\right|}
\newcommand{\set}[1]{\left({#1}\right)}
\newcommand{\eps}{\varepsilon}
\newcommand{\mylabel}[1]{\label{#1}}  

\newenvironment{myEnumerate}
{\begin{enumerate}
  \setlength{\itemsep}{3pt}
  \setlength{\parskip}{0pt}
  \setlength{\parsep}{0pt}}
{\end{enumerate}}

\begin{document}
\maketitle\thispagestyle{empty}


\begin{summary}
   In this paper, we study the summability properties of double sequences of real constants
   which map sequences of random variables to sequences of random variables that are defined on the same probability sample space.
   We show that a regular method of summability is still regular on sequences of random variables with almost everywhere convergence,
   almost sure convergence, and with $L_p$-convergence. 
   It is not necessarily regular on sequences of random variables with convergence in probability.
   We extend these results to random variables with values in extended real numbers
  (extended real numbers include infinite values, see definitions \ref{def:extendedReal} and \ref{def:1}).
  For this we introduce a construction that allows us to multiply sequences of extended real numbers with infinite real matrices.
\end{summary}

\renewcommand{\thefootnote}{}
\footnotetext{\hspace*{-.51cm}
}

\section{Introduction}\mylabel{sec:1}

Let $ A = (a_{ij}),\; i = 1, 2, ... ,\; j = 0, 1, 2, ... , $ be a double sequence of real constants, that is,
\begin{eqnarray}
  { A={\left(\begin{matrix}a_{10} & a_{11} & a_{12} & \text{......}\\
  a_{20} & a_{21} & a_{22} & \text{......}\\
  a_{30} & a_{31} & a_{32} & \text{......}\\
  \text{...} & \text{...} & \text{...} & \text{......}\end{matrix}\right)} }
  .
\end{eqnarray}

Let $ S $ denote the set of sequences of real constants. 
In traditional summability theory, $A$ is, under certain assumptions, 
considered to be a mapping from a subset of $ S $ to $ S $ 
(see \cite{Baron:66}, \cite{Brudno:45}, \cite{Cooke:50}, \cite{Fast:51},
 [14-18],
 \cite{Peyerimhoff:69}).

Let $S_1, S_2$ be non-empty subsets of $S$. $A$ is said to be summable from $S_1$ to $S_2$, 
whenever for any $x\in S_1, Ax\in S_2$. 
It has been studied by many authors for $S_1$ to be some special space contained in $S$.
For example, $S_1=\ell_p, S_b,$ etc, where $S_b$ denotes the set of bounded sequences of real numbers 
(see \cite{Boos:00},
 [10-13],
 \cite{Neubrum:Smital:Salat:68}, \cite{Zygmund:41}).

In recent years, some authors have extended the concept of summability to statistical summability, 
which studies the convergence of sequences mapped by $A$ in a certain proportion 
(see \cite{Belen:Mursaleen:Yildirim:12}, \cite{Connor:88},
 [7-8],
 [19-22],
 [25-26]).

In probability theory and stochastic processes theory, 
the convergence of a sequence of random variables has been an important topic, e.g. in central limit theorems.
It leads us to consider the summability of the sequence of random variables by a given double sequence of real numbers $A$.

\bigskip
Let $ (\Omega, F, P) $ be a probability space.  Let $ x = [X_1, X_2, X_3, ...] = \{X_n\}_{n=1}^\infty$ 
be a sequence of real valued random variables defined on the same sample space $ \Omega $. 
We will consider the following types of convergence of this sequence of random variables:
\begin{myEnumerate}
  \item[(1)]    $ X_n \rightarrow X_\infty, $  a.e. in $ \Omega $;
  \item[(2)]    $ X_n \rightarrow X_\infty, $ in pr.;
  \item[(3)]    $ X_n \rightarrow X_\infty, $ in $ L_p(\Omega), $ for $ p \ge 1 $.
\end{myEnumerate}
In this paper, we are interested in finding the conditions on $ A $, 
under which $ Ax $ is convergent a.e., almost surely, in pr., or in $ L_p(\Omega), $ respectively,
for any given $ x $ satisfying one of the above conditions.

\section{Some known results and a need for new approach}\mylabel{sec:2}

All random variables considered in this paper are real random variables that are finite almost everywhere 
and defined on a probability space $ (\Omega, P, F) $, that is, every random variable $ X $, satisfies
\begin{eqnarray}\mylabel{eq:finite}
  P(X \in \R) = 1.
\end{eqnarray}

For example, the following results are straightforward
consequences from the traditional summability theory
(\cite{Baron:66}, \cite{Cooke:50}, \cite{Peyerimhoff:69}).

\begin{proposition}\mylabel{prop:1}  If $A$ satisfies the following three conditions:
  \begin{enumerate}
    \item ${\underset{1\le i<\infty}{\text{lub}}{{\sum}^{\infty}_{j=0}}{ |a_{ij}| }=M<\infty}$;  
    \item ${\underset{i\rightarrow\infty}{\text{lim}} \,a_{ij} } \;\;\text{ exists for } j = 1, 2, \ldots{}$;
    \item ${\underset{i\rightarrow\infty}{\text{lim}}{\sum}^{\infty}_{j=1}{ a_{ij} }} \;\;\text{ exists}$.
  \end{enumerate}
  Then $ X_n \rightarrow X_\infty, $  a.e. in $ \Omega $ implies  
  $ (Ax)_n \rightarrow \tilde{X}_\infty, $  a.e. in $ \Omega $ for some random variable $ \tilde{X}_\infty $ 
  that may be different from $ X_\infty $.
\end{proposition}

Extended real numbers are often defined as:

\begin{definition}\mylabel{def:extendedReal}
  \begin{eqnarray}
    \mathbb{R}^* = \mathbb{R} \cup \{-\infty,\infty\}.
  \end{eqnarray}
\end{definition}

However, to handle matrix multiplication $ Ax $, where the vector $ x $ can have infinite entries, 
we need to be able to add and subtract infinite values.

\bigskip
Let $ V $ be a 2-dimensional vector space
over $ \mathbb{R} $ with basis $ (1,x). $
Its lexicographic ordering generated by $ 1 < x $ and the standard ordering of $ \mathbb{R} $ 
is consistent with the vector addition and multiplication by a scalar:

Let $ u,v,w \in V \wedge a \in \mathbb{R}, $ then:
\begin{eqnarray*}
  u < v  \Rightarrow & u + w & < v + w \\
  (0 < a) \wedge (u < v)  \Rightarrow & a \cdot u & < a \cdot v
\end{eqnarray*}

\begin{definition}\mylabel{def:1}
  In the above construction, make the following notation:
  \begin{eqnarray*}
    \infty & := & x \\
    \Rii & := & V \\
    \R_{inf}^+ & := & \{u \in \Rii:\; \forall a \in \R\; (a < u) \} \\
    \R_{inf}^- & := & \{u \in \Rii:\; \forall a \in \R\; (u < a) \} \\
    \R_{inf} & := & \R_{inf}^+ \cup \R_{inf}^-
  \end{eqnarray*}
  Then we call $ \Rii $ the {\it extended real numbers} and $ \R_{inf} $ the {\it infinite real numbers}.
\end{definition}

\begin{remark}\mylabel{rem:1}
  In this new notation, $  \Rii =   \R \cup \R_{inf} $ and every $ u \in \Rii $ 
  can be written in a unique way as $ u=a+b \cdot \infty $ for some $ a,b \in \R. $
\end{remark}

\begin{definition}\mylabel{def:2}
  For any  $ u=a+b \cdot \infty \in \Rii $ define the {\it absolute value} of $ u $ as  $ |u|:=|a|+|b| \cdot \infty. $
\end{definition}

\begin{remark}\mylabel{rem:2}
  There is a cannonical projection from the set of random variables $ X: \Omega \rightarrow \Rii $ 
  to the set of random variables $ X: \Omega \rightarrow \mathbb{R}^* $. Each $X(\omega)$ is mapped by the identity on $\R$, 
  positive infinite values (ordering is a part of the construction preceding definition \ref{def:1}) 
  are mapped to $+\infty$ and negative infinite values to $-\infty$.
  Considering of this projection allows us to extend proofs of all theorems 
  for random variables $ X: \Omega \rightarrow \mathbb{R}^* $ to random variables $ X: \Omega \rightarrow \Rii $. 
  More precisely, random variables $ X: \Omega \rightarrow \Rii $ 
  form a model of the theory of random variables $ X: \Omega \rightarrow \mathbb{R}^* $.
\end{remark}

From now on consider random variables $ X: \Omega \rightarrow \Rii $.
The following extension of Integral and Expected values will be sufficient for us.

\begin{definition}\mylabel{def:3}
  Let $ X $ be a random variable that is finite a.e. in $ \Omega $. Then
  $$  \ds{ E(X) = \int_\Omega X(\omega) P(d \omega) := \int_{ \{\omega: X(\omega) \in \R\} } X(\omega) P(d \omega) }. $$
\end{definition}

\begin{lemma}\mylabel{lem:1}
  \begin{myEnumerate}
    \item[(a)]  $ \R_{inf} = \displaystyle\bigcap_{K \in \mathbb{N}} \{u \in \Rii:\; |u| > K\},\;\; $
    \item[(b)]  $ P(X \in \R_{inf}) = \displaystyle\lim_{K\to\infty} P(|X| > K). $  
  \end{myEnumerate}
\end{lemma}

\begin{proof}
  From definitions \ref{def:1} and \ref{def:2} we have $ \R_{inf} = \{u \in \Rii:\; \forall a \in \R\; (a < |u|) \}, $ 
  which is equivalent to (a). So  $ P(X \in \R_{inf}) $
  \br$
    = P\left(\displaystyle\bigcap_{K \in \mathbb{N}} \{\omega:\; X(\omega) \in \Rii,\; |X(\omega)| > K\}\right) 
    = P\left(\displaystyle\bigcap_{K \in \mathbb{N}} \{\omega:\; |X(\omega)| > K\}\right).
  $
  \br Since $ \left(\{\omega: |X(\omega)| > K\}\right)_{k=1}^\infty $ is a decreasing sequence w.r.t. inclusion, we have
  \br$
    P\left(\displaystyle\bigcap_{K \in \mathbb{N}} \{\omega: |X(\omega)| > K\}\right)
    = \displaystyle\lim_{K\to\infty} P\left(\{\omega: |X(\omega)| > K\}\right)
    = \displaystyle\lim_{K\to\infty} P(|X| > K).
  $
\end{proof}

\bigskip
The following claim is well known and we include it here as a lemma without proof.

\begin{lemma}\mylabel{lem:2}
  \begin{myEnumerate}
    \item[(a)]  If $ \Omega' \subseteq \Omega $ and $ P(\Omega')=1 $ then $ X_n \to X_\infty $\,, in $ L^p(\Omega) $ 
      is equivalent to  $ X_n \to X_\infty $\,, in $ L^p(\Omega') $,
    \item[(b)]  if countably many random variables $ X_n $ are finite a.e. in $ \Omega $ then 
      there is $ \Omega' \subseteq \Omega $ with $ P(\Omega')=1 $, such that all $ X_n $ are finite in $ \Omega' $.
  \end{myEnumerate}
\end{lemma}

The following proposition shows that a sequence $ x = \{X_n\}_{n=1}^\infty$ of r.v.'s converging in pr. 
does not imply that $ Ax $ converges in pr..    

\begin{proposition}\mylabel{prop:3} Suppose that A satisfies the following conditions:
  \begin{enumerate}
    \item ${\underset{1\le i<\infty}{\text{lub}}{{\sum}^{\infty}_{j=0}}{ |a_{ij}| }=M<\infty}$;  
    \item ${\underset{i\to\infty}{\text{lim}} a_{ij} } = 0 \text{ for } j = 1, 2, \ldots{}$;
    \item ${\underset{i\to\infty}{\text{lim}}{\sum}^{\infty}_{j=1}{ a_{ij} }} = 1$.
  \end{enumerate}
  Then  $ X_n \to X_\infty $  in pr. does not imply   $ (Ax)_n \to X_\infty $ in pr.
\end{proposition}

This is a well known result in probability theory. To show it, for the convenience of the reader, we provide an example below.

\begin{example}\mylabel{ex:1}
  We take interval $ [0, 1) $ as the sample space. If $ n=2^m+i $, for some given $ m = 1, 2, 3, ..., $ 
  and for some $ 0 \le i < 2^m $, then we define $ X_n $ as follows:
  \begin{eqnarray}
    X_n(\omega) = \left\{
    \begin{array}{ll}
      4^{m+i}, & \text{ if } \omega \in \left[\frac{i}{2^m},\frac{i+1}{2^m}\right) \\
      0, &  \text{ otherwise} \\
    \end{array}
    \right.
  \end{eqnarray}
  It is clear that $ X_n \to 0 $ in pr.
  Taking the Cesaro Summability method $ A $, for $ n \ge 16 $ and $ n=2^m+i $, 
  where $ 0 \le i < 2^m $, noting $ m>3 $ and so $ 4^{m-1} > 2^{m+1} $, we have

  \begin{eqnarray*}
    \left(\frac{\sum_{j=1}^n X_j}{n} > 1 \right)
      &=& \left(\sum_{j=1}^n X_j > n \right)
      \supseteq \left(\sum_{j=1}^n X_j > 2^{m+1} \right)
      \\
      &\supseteq& \left(\sum_{j=2^{m-1}}^{2^m - 1} X_j > 2^{m+1} \right)
      = \left(\sum_{i=0}^{2^{m-1} - 1} X_{2^{m-1} + i} > 2^{m+1} \right)
      \\
      &=& \bigcup_{i=0}^{2^{m-1} - 1} \left[\frac{i}{2^{m-1}},\frac{i+1}{2^{m-1}}\right) = [0,1).
  \end{eqnarray*}

  \noindent It implies $ \ds{ P\left(\frac{\sum_{j=1}^n X_j}{n} > 1 \right) = 1 } $ and shows that 
   $ \ds\lim_{n \to \infty}P\left(\frac{\sum_{j=1}^n X_j}{n} > 1 \right) = 1 > 0.$

  \noindent Hence $ \displaystyle{ \frac{\sum_{j=1}^n X_j}{n} } $ does not converge to $ 0 $ in pr..
\end{example}

\bigskip
The above example demonstrates that even with Cesaro Summability method $ A $, 
$ X_n \to X_\infty $ in pr. does not imply that $ (Ax)_n \to X_\infty $ in pr.. 
Hence to assure $ (Ax)_n \to X_\infty $ in pr. 
for any given regular summability method $ A $, 
a stronger condition on the sequence $ (X_n) $  than $ X_n \to X_\infty $ in pr. is needed.

In all the following,  ${X_n,\; n=1,2,...}$ and ${X_\infty}$ 
is a sequence of random variables on the same probability space ${\Omega}$ and $x=(X_1,X_2, \dots{})$.
We review the definition of \emph{almost sure covergence} now.

\begin{definition}\mylabel{def:asConv}
  ${ X_n \suparrow X_\infty }$  if for any ${\lambda > 0}$
  \br${ \displaystyle\lim_{n\to\infty}\, P\left( \displaystyle\sup_{n \le m < \infty} |X_m - X_\infty| > \lambda \right) = 0}$.
  
\end{definition}

We have easily the following corollary, which is well-known.

\begin{corollary}\mylabel{cor:3}
  ${ X_n \suparrow X_\infty}$  implies $ {X_n \rightarrow X_\infty }$  in pr.
\end{corollary}

\begin{example}\mylabel{ex:2}
  Let set $ S $ is such that $ P(S)=\eps > 0 $ and 
  let $ X_1(\omega) = \infty $ for $ \omega \in S $ and zero otherwise,
  $ X_n(\omega) = 0 $ everywhere for $ n = 2,3, ... , \infty $.
  Here $ X_n \to X_\infty $ even in the strongest sense one can think of but,
  for $ A $ that has all elements in the first column nonzero,
  we still don't have $ (Ax)_n \to X_\infty $.

  Here $ X_1 $ doesn't satisfy condition (\ref{eq:finite}).
  While in our extended real numbers, we are able to do algebra with infinities,
  this example shows that the condition (\ref{eq:finite}) is still needed.
\end{example}

\section{Some new stochastic summability results}\mylabel{sec:3}

\begin{proposition}\mylabel{prop:2}
  Let $ X_n,\, X_\infty $  be finite a.e., that is $ P(X_n \in \R)=1, P(X_\infty \in \R)=1, $
  and  $ A $ define a regular method of summability, that is
  \begin{enumerate}
  \item ${\underset{1\le i<\infty}{\text{lub}}{{\sum}^{\infty}_{j=0}}{ |a_{ij}| }=M<\infty}$;  
  \item ${\underset{i\rightarrow\infty}{\text{lim}} a_{ij} } = 0 \text{ for } j = 1, 2, \ldots{}$;
  \item ${\underset{i\rightarrow\infty}{\text{lim}}{\sum}^{\infty}_{j=1}{ a_{ij} }} = 1$.
  \end{enumerate}
  Then  $ X_n \rightarrow X_\infty $  a.e. in $ \Omega $ implies 
  $ (Ax)_n \rightarrow X_\infty $  a.e. in $ \Omega $.
\end{proposition}

\begin{proof}  
  From Lemma \ref{lem:2}(b) and the assumptions, we have the existence of a sequence of sets $ S_n $ 
  and sets $ S_\infty, S: $
  \begin{eqnarray*}
    P(|X_n| \in R)=1 &\Rightarrow& \exists S_n \subseteq \Omega:\; P(S_n)=1 \wedge (\omega \in S_n \Rightarrow |X_n(\omega)| \in \R)\\
    P(|X_\infty| \in R)=1 &\Rightarrow& \exists S_\infty \subseteq \Omega:\; P(S_\infty)=1 \wedge (\omega \in S_\infty \Rightarrow |X_\infty(\omega)| \in \R)\\
    P(X_n \rightarrow X_\infty)=1 &\Rightarrow& \exists S \subseteq \Omega:\; P(S)=1 \wedge (\omega \in S \Rightarrow X_n(\omega) \rightarrow X_\infty(\omega))
  \end{eqnarray*}
  Consider $ T_n = \bigcap_{i=1}^n S_i \cap S \cap S_\infty $. 
  The sequence $ (T_n)_{i=1}^\infty $ is decreasing w.r.t. inclusion, 
  so $ P(T) = P(\bigcap T_n) = \lim_{n \rightarrow \infty} P(T_n)=1 $.
  Now, for every $ \omega \in T $ we have $ |X_n(\omega)| \in \R $ and $ |X_\infty(\omega)| \in \R $. 
  So the Silverman-Toeplitz theorem applies for each $ n $ 
  and consequently $ (Ax)_n \rightarrow X_\infty $  a.e. in $ \Omega $.
\end{proof}

\begin{proposition}\mylabel{prop:4}
  Suppose that A is a regular method of summability and its norm $|A|=M$. 
  That means, by Silverman-Toeplitz Theorem:
  \begin{enumerate}
    \item ${\underset{1\le i<\infty}{\text{lub}}{{\sum}^{\infty}_{j=0}}{ |a_{ij}| }=M<\infty}$;  
    \item ${\underset{i\rightarrow\infty}{\text{lim}} a_{ij} } = 0 \text{ for } j = 1, 2, \ldots{}$;
    \item ${\underset{i\rightarrow\infty}{\text{lim}}{\sum}^{\infty}_{j=1}{ a_{ij} }} = 1$.
  \end{enumerate}
  Then  $ X_n \suparrow X_\infty $  implies $ (Ax)_n \suparrow X_\infty $\,, whenever
  $$  X_n,\, X_\infty  \text{ are finite a.e., that is, } P(X_n \in \R)=1, P(X_\infty \in \R)=1. $$
\end{proposition}
\begin{proof}
  By Lemma \ref{lem:1} and from finiteness a.e., we have conditions
  \begin{myEnumerate}
    \item[(*)] ${\displaystyle\lim_{K\to\infty} P(|X_n| > K) = 0,\; \displaystyle\lim_{K\to\infty} P(|X_\infty| > K) = 0}$
  \end{myEnumerate}
  For any given ${ \varepsilon,\delta > 0}$, we have to show that there exists ${N > 1}$,
  such that for all ${n > N}$, the following inequality holds
  \begin{equation*}
    P\left( \displaystyle\sup_{n \le i < \infty} |(Ax)_i - X_\infty| > \delta \right) < \varepsilon.
  \end{equation*}
  For the given ${\varepsilon,\delta > 0}$,
  $ X_n \suparrow X_\infty $  implies that  there exists ${ N_2 > 1}$, such that
  \begin{equation}\mylabel{eqXsuparrow}
    P\left( \displaystyle\sup_{n \le k < \infty} |X_k - X_\infty| > \frac{\delta}{3M} \right) < \frac{\varepsilon}{2}, \text{ for all } n \ge N_2.
  \end{equation}
  From conditions (*), for the already known $\varepsilon$ and $N_2$ there exists ${K > 1}$, such that
  \begin{equation}\mylabel{eq1}
    P\left( \max_{\ontop{1 \le k < N_2}{or\,k=\infty}} |X_k| > K \right) < \frac{\varepsilon}{2}
  \end{equation}
  For this fixed ${K > 1}$, we have
  \begin{eqnarray}\mylabel{eqsupAx}
  && { P\left( \sup_{n \le i < \infty} |(Ax)_i - X_\infty| > \delta \right) \hphantom{ aa } } \nonumber \\
  & = &  P\left( \sup_{n \le i < \infty} |(Ax)_i - X_\infty| > \delta:\;  \max_{\ontop{1 \le k < N_2}{or\,k=\infty}} |X_k| \leq K  \right)
    \\ & & \hspace{2.5in} +  P\left( \sup_{n \le i < \infty} |(Ax)_i - X_\infty| > \delta:\;  \max_{\ontop{1 \le k < N_2}{or\,k=\infty}} |X_k| > K  \right)  \nonumber \\
  & < &  P\left( \sup_{n \le i < \infty} |(Ax)_i - X_\infty| > \delta:\;  \max_{\ontop{1 \le k < N_2}{or\,k=\infty}} |X_k| \leq K  \right)
      + \frac{\varepsilon}{2}.
  \end{eqnarray}
  For the fixed $N_2$, there exists $N_3 \ge N_2$, from condition 2. in this proposition, such that
  \begin{equation*}
    |a_{nj}| \le \frac{\delta}{6KN_2}, \text{ for all } n \ge N_3, \text{ and all } 1 \le j < N_2.
  \end{equation*}
  From condition 3. in this proposition, there exists $N \ge N_3$, such that
  \begin{equation*}
    \abs{ \sum_{j=1}^{\infty} a_{nj} - 1 } \le \frac{\delta}{3K}, \text{ for all } n \ge N, \text{ and so }
    \sup_{n \le i < \infty} \abs{ \sum_{j=1}^{\infty} a_{ij} - 1 } \le \frac{\delta}{3K}.
  \end{equation*}
  Now, for all $n \ge N$, we have
  \begin{eqnarray*}\mylabel{eq3}
  && { \set{ \sup_{n \le i < \infty} \abs{(Ax)_i - X_\infty} > \delta:\;  \max_{\ontop{1 \le k < N_2}{or\,k=\infty}} |X_k| \leq K  } } \\
    & = & \set{ \sup_{n \le i < \infty}  \abs{
	\sum_{j=1}^{\infty} a_{ij} X_j - \sum_{j=1}^{\infty} a_{ij} X_\infty
	+ \left( \sum_{j=1}^{\infty} a_{ij} -1 \right) X_\infty
      }  > \delta:\;  \max_{\ontop{1 \le k < N_2}{or\,k=\infty}} |X_k| \leq K  } \\
    & = & \set{ \sup_{n \le i < \infty} \abs{
	\sum_{j=1}^{\infty} a_{ij} (X_j - X_\infty)
	+ \left( \sum_{j=1}^{\infty} a_{ij} -1 \right) X_\infty
      }  > \delta:\;  \max_{\ontop{1 \le k < N_2}{or\,k=\infty}} |X_k| \leq K  } \\
    & = & \set{ \sup_{n \le i < \infty} \abs{
	\sum_{j=1}^{N_2-1} a_{ij} (X_j - X_\infty) + \sum_{j=N_2}^{\infty} a_{ij} (X_j - X_\infty) 
	+ \left( \sum_{j=1}^{\infty} a_{ij} -1 \right) X_\infty
      } > \delta:\;  \max_{\ontop{1 \le k < N_2}{or\,k=\infty}} |X_k| \leq K  } \\
    & \subseteq & \left[
	\set{ \sup_{n \le i < \infty}  \abs{ \sum_{j=1}^{N_2 - 1} a_{ij} (X_j - X_\infty) } > \frac{\delta}{3} }
	\bigcup \set{ \sup_{n \le i < \infty}  \abs{ \sum_{j=N_2}^{\infty} a_{ij} (X_j - X_\infty) } > \frac{\delta}{3} }
      \right. \\
    & &   \left.  \bigcup \set{ \sup_{n \le i < \infty}  \abs{ \left( \sum_{j=1}^{\infty} a_{ij} -1 \right) X_\infty } > \frac{\delta}{3} }
      \right]  \bigcap \set{ \max_{\ontop{1 \le k < N_2}{or\,k=\infty}} |X_k| \leq K } \\
    & \subseteq & \left[
	\set{ \sup_{n \le i < \infty}  \sum_{j=1}^{N_2 - 1} |a_{ij}| (|X_j| + |X_\infty|)  > \frac{\delta}{3} }
	\bigcup \set{ \sup_{n \le i < \infty}  \sum_{j=N_2}^{\infty} |a_{ij}| \abs{X_j - X_\infty} > \frac{\delta}{3} }
      \right. \\
    & &   \left.  \bigcup \set{ \sup_{n \le i < \infty}  \abs{ \left( \sum_{j=1}^{\infty} a_{ij} -1 \right) } |X_\infty| > \frac{\delta}{3} }
      \right]  \cap \;\Omega \\ 
    & \subseteq & \left[
	\set{ \sup_{n \le i < \infty}  \sum_{j=1}^{N_2 - 1} \frac{\delta}{6KN_2} 
	    \left( \max_{1 \le k < N_2} |X_k| + |X_\infty| \right)  > \frac{\delta}{3} } 
      \right. \\
    & &   \left.  \bigcup \set{ \sup_{n \le i < \infty}  \sum_{j=N_2}^{\infty} |a_{ij}|  \sup_{N_2 \le k < \infty} \abs{X_k - X_\infty} > \frac{\delta}{3} }
      \bigcup \set{  \frac{\delta}{3K} |X_\infty| > \frac{\delta}{3} }
      \vphantom{\sum_{j=1}^{N_2 - 1}} \right] \\ 
    & \subseteq & \left[
	\set{ \left(\max_{1 \le k < N_2} |X_k| + |X_\infty| \right)  > 2K }
	\right. \left. \cup \set{   \sup_{N_2 \le k < \infty} \abs{X_k - X_\infty} > \frac{\delta}{3M} }
	\cup \set{ |X_\infty| > K  \vphantom{\max_{1 \le k < N_2}} }
      \right]  \\  
    & &  \\ 
    & \subseteq & \;\; \emptyset \cup \set{   \sup_{N_2 \le k < \infty} \abs{X_k - X_\infty} > \frac{\delta}{3M} } \cup \emptyset.
  \end{eqnarray*}
  Applying probability to these sets and by (\ref{eqXsuparrow}) we have: 
  \begin{eqnarray*}
    P\set{ \sup_{n \le i < \infty} \abs{(Ax)_i - X_\infty} > \delta:\;  \max_{\ontop{1 \le k < N_2}{or\,k=\infty}} |X_k| \leq K  }
    < P\set{   \sup_{N_2 \le k < \infty} \abs{X_k - X_\infty} > \frac{\delta}{3M} }
    < \frac{\varepsilon}{2}
  \end{eqnarray*}
  Now going back to (\ref{eqsupAx}) we can finish our proof:
  \begin{equation*}
    P\left( \sup_{n \le i < \infty} |(Ax)_i - X_\infty| > \delta \right)
    < \frac{\varepsilon}{2} + \frac{\varepsilon}{2} = \varepsilon.
  \end{equation*}
\end{proof}

If we take the Cesaro Summability method $ A $, then we have the following corollary of Proposition \ref{prop:4}.

\begin{corollary}\mylabel{cor:4}  
  For $ n = 1, 2, ..., $ let random variables $ X_n, X_\infty $ satisfy
  \begin{myEnumerate}
    \item[(a)]  $ X_n,\, X_\infty $  are finite a.e.,
    \item[(b)]  $ X_n \suparrow X_\infty $ .  
  \end{myEnumerate}
  Then $ \displaystyle{ \frac{\sum_{i=1}^n X_i}{n} \suparrow X_\infty } $ . 
\end{corollary}

\begin{proposition}\mylabel{prop:5}
  Let $ p \ge 1 $. Suppose that $ A $ satisfies the following conditions:
  \begin{enumerate}
    \item ${\underset{1\le i<\infty}{\text{lub}}{{\sum}^{\infty}_{j=0}}{ |a_{ij}| }=M<\infty}$;  
    \item ${\underset{i\rightarrow\infty}{\text{lim}} a_{ij} } = 0 \text{ for } j = 1, 2, \ldots{}$;
    \item ${\underset{i\rightarrow\infty}{\text{lim}}{\sum}^{\infty}_{j=1}{ a_{ij} }} = 1$.
  \end{enumerate}
  Then $ X_n \to X_\infty $\,, in $ L^p(\Omega) $, implies  $ (Ax)_n \to X_\infty $\,, in $ L^p(\Omega) $, whenever  
  $$  X_n,\, X_\infty  \text{ are finite a.e., that is, } P(X_n \in \R)=1, P(X_\infty \in \R)=1. $$
\end{proposition}

\begin{proof}
  For any given $ \eps > 0 $, the condition $ X_n \to X_\infty $\,, in $ L^p(\Omega) $ (refer to Definition \ref{def:3} and Lemma \ref{lem:2}) implies 
  that there exists $ N_1 > 1 $, such that for all $ n \ge N_1 $, 
  the following inequality holds
  \begin{eqnarray}\mylabel{eq5:2}
    (E|X_n - X_\infty|^p)^{\frac{1}{p}} < \frac{\eps}{3M}.
  \end{eqnarray}
  The condition   $ X_n,X_\infty \in L^p(\Omega) $ implies that there exists $ K > 1 $, 
  such that $ ||X_n||_p \le K $ (again refer to Definition \ref{def:3}) for all $ n = 1, 2, ..., \infty $. 
  For the fixed $ N_1 $. Condition 2. implies that there exists $ N_2 \ge N_1 $, such that
  \begin{eqnarray}\mylabel{eq5:3}
    |a_{nj}| \le \frac{\eps}{6KN_1}, \text{ for all } n \ge N_2, \text{ and all } 1 \le j \le N_1.
  \end{eqnarray}
  From the condition 3. in this proposition, there exists $ N \ge N_2 $, such that
  \begin{eqnarray}\mylabel{eq5:4}
    \left| \sum_{j=1}^\infty a_{nj} -1 \right| \le \frac{\eps}{3K}, \text{ for all } n \ge N.
  \end{eqnarray}
  Now, for all $ n \ge N $, from  (\ref{eq5:2}), (\ref{eq5:3}), and (\ref{eq5:4}), we have
  \begin{eqnarray*}
    & (E|(Ax)_n - X_\infty|^p)^{\frac{1}{p}} = \\
    & \ds{ = \left( E \abs{ \sum_{j=1}^{\infty} a_{nj} X_j - \sum_{j=1}^{\infty} a_{nj} X_\infty + \left( \sum_{j=1}^{\infty} a_{nj} -1 \right) X_\infty }^p \right)^{\frac{1}{p}} } \\
    & \ds{ \le \left( E \abs{ \sum_{j=1}^{\infty} a_{nj} (X_j - X_\infty)  }^p \right)^{\frac{1}{p}}   +   \left( E \abs{ \left( \sum_{j=1}^{\infty} a_{nj} -1 \right) X_\infty }^p \right)^{\frac{1}{p}} } \\
    & \ds{ \le \left( E \abs{ \sum_{j=1}^{N_1} a_{nj} (X_j - X_\infty)  }^p \right)^{\frac{1}{p}}    +   \left( E \abs{ \sum_{j=N_1 +1}^{\infty} a_{nj} (X_j - X_\infty)  }^p \right)^{\frac{1}{p}}   +   \left( E \abs{ \left( \sum_{j=1}^{\infty} a_{nj} -1 \right) X_\infty }^p \right)^{\frac{1}{p}} } \\
    & \ds{ \le  \sum_{j=1}^{N_1} \abs{a_{nj}} \left( E \abs{ (X_j - X_\infty)  }^p \right)^{\frac{1}{p}}    +   \sum_{j=N_1 +1}^{\infty} \abs{a_{nj}}  \left( E \abs{ (X_j - X_\infty)  }^p \right)^{\frac{1}{p}}   +    \abs{ \left( \sum_{j=1}^{\infty} a_{nj} -1 \right) } \left( E \abs{X_\infty}^p \right)^{\frac{1}{p}} } \\
    & \ds{ \le  2K \sum_{j=1}^{N_1} \abs{a_{nj}}   +   \sum_{j=N_1 +1}^{\infty} \abs{a_{nj}} \frac{\eps}{3M}   +    \frac{\eps}{3K} K } \\
    & < \eps.
  \end{eqnarray*}
\end{proof}

Similarly to the proof of Proposition \ref{prop:5}, by applying the completeness property of $ L^p(\Omega) $, 
we can prove the following proposition.

\begin{proposition}\mylabel{prop:6}
  Let $ p \ge 1 $. If $ A $ satisfies the following conditions:
  \begin{enumerate}
    \item ${\underset{1\le i<\infty}{\text{lub}}{{\sum}^{\infty}_{j=0}}{ |a_{ij}| }=M<\infty}$;  
    \item ${\underset{i\rightarrow\infty}{\text{lim}} a_{ij} } = 0 \text{ for } j = 1, 2, \ldots{}$;
    \item ${\underset{i\rightarrow\infty}{\text{lim}}{\sum}^{\infty}_{j=1}{ a_{ij} }} = 1$.
  \end{enumerate}
  Then  $ x $ converges in $ L^p(\Omega) $ implies that  $ Ax $  converges in $ L^p(\Omega) $, whenever 
  $$  X_n,\, X_\infty  \text{ are finite a.e., that is, } P(X_n \in \R)=1, P(X_\infty \in \R)=1. $$
\end{proposition}

\begin{remark}\mylabel{rem:6}
  By following the steps in the proof of Proposition \ref{prop:5}, 
  one can show that both Propositions \ref{prop:5} and \ref{prop:6} hold for $ p = \infty $.
\end{remark}



\bibliographystyle{plain}

\bigskip
\noindent
\parbox[t]{.48\textwidth}{
Jinlu Li \\
Department of Mathematics \\
Shawnee State University \\
940 Second Street \\
Portsmouth, OH 45662, USA \\
jli@shawnee.edu 
} \hfill
\parbox[t]{.48\textwidth}{
Robert Mendris \\
Department of Mathematics \\
Shawnee State University \\
940 Second Street \\
Portsmouth, OH 45662, USA \\
rmendris@shawnee.edu 
} \hfill

\end{document}